\documentclass[12pt]{article}
\usepackage{indentfirst}
\usepackage{geometry,graphicx,xcolor,color,enumerate,verbatim}
\usepackage{amssymb,amsmath,amsthm,mathtools,bm}                           
\usepackage{mathpazo}
\usepackage[nofontspec]{newpxtext}

\definecolor{winered}{rgb}{0.5,0,0}
\definecolor{PaleTurquoise}{RGB}{102 ,139, 139}
\definecolor{structurecolor}{RGB}{122,122,142}
\definecolor{main}{rgb}{0.5,0,0}
\definecolor{second}{RGB}{115,45,2}
\definecolor{third}{RGB}{0,80,80}
\usepackage[colorlinks,linkcolor = winered, citecolor=winered]{hyperref}           

\usepackage{newclude,ulem}
\usepackage{amsthm}
\newtheoremstyle{defstyle}{10pt}{20pt}{ \itshape }{-3pt}{\bfseries\color{main}}{}{0.1em}{ \underline{\thmname{#1} \thmnumber{#2}. \thmnote{(#3)}} }
\newtheoremstyle{thmstyle}{10pt}{20pt}{ \itshape }{-3pt}{\bfseries\color{second}}{}{0.1em}{ \uline{\thmname{#1} \thmnumber{#2}. \thmnote{(#3)}} }
\newtheoremstyle{prostyle}{10pt}{20pt}{ \itshape }{-3pt}{\bfseries\color{third}}{}{0.1em}{ \uline{\thmname{#1} \thmnumber{#2}. \thmnote{(#3)}}  }

\theoremstyle{thmstyle}
    \newtheorem{theorem}{Theorem}[section]
\theoremstyle{defstyle} 
    
    \newtheorem{lemma}[theorem]{Lemma}
    \newtheorem{corollary}[theorem]{Corollary}
\theoremstyle{prostyle}
    \newtheorem{proposition}[theorem]{Proposition}
    
    \newtheorem{remark}[theorem]{Remark}

\renewenvironment{proof}[1][Proof.]{\par{\noindent \uline{\bfseries #1}} \; }{\qed\par}

\numberwithin{equation}{section}

\title{Uniform heat kernel asymptotics for the Grushin operator}
\author{ Yimeng Chen, Hong-Quan Li, Jun-Cheng Tang, Jia-Yu Yang }
\date{\today}

\def\R{\mathbb{R}}
\def\C{\mathbb{C}}

\def\wtreffun{\widetilde{\phi}}
\def\reffun{\phi}

\def\ff{\mathrm{f}}
\def\wtff{\widetilde{\ff}}
\def\ra{R,a}
\def\rr{\mathrm{r}}
\def\bV{\mathbf{V}}
\def\bVt{\bV_2}
\def\ggg{g,g'}
\def\yg{\mathbf{y}_{\ggg}}
\def\Sg{\mathbf{S}_{\ggg}}
\def\Wg{\mathbf{W}_{\ggg}(\xi)}
\def\F{\mathbf{F}}
\def\d{\,d}

\def\Ag{\mathbb{A}_{\ggg}(\xi)}
\def\hess{\mathrm{Hess}}
\def\t{^\mathrm{T}}
\def\I{\mathbb{I}}
\def\xg{\mathbf{x}_{\ggg}}

\def\vt{\vartheta}
\def\cpi{\mathbf{C}(n, n')}
\def\tcpi{\tilde{\mathbf{C}}(n, n')}

\def\rgn{\blacklozenge}
\def\O{\mathcal{O}}
\def\rg{\rho_{\ggg}(\xi)}
\def\Ic{\mathcal{I}}
\def\Yy{\mathbf{Y}}
\def\BR{\mathbf{R}}
\def\ss{\mathrm{s}}

\def\gx{\rho^{*}(\ggg;\xi)}
\def\adr{\lozenge}

\def\pa{\psi_a}
\def\ma{\mu_a}

\def\AAg{\widetilde{\mathbf{A}}_{\ggg}(\eta)}
\def\EEg{\widetilde{\mathbf{E}}_{\ggg}(\eta)}
\def\EgE{\mathbf{E}_{\ggg}(\xi)}
\def\Lgg{\Lambda_{\ggg}(\xi)}
\def\vpo{\varpi_0}
\def\DD{\mathrm{D}}
\def\sV{\tilde{\mathbf{V}}}
\def\jj{\mathbf{J}}
\def\wtjj{{\widetilde{\jj}}}

\DeclareMathOperator{\csch}{csch}

\begin{document}
    \maketitle

\begin{abstract}
Adopting the powerful methods introduced in \cite{li2021carnotcaratheodory, LZ2025},  we investigate the asymptotic behaviour at infinity for the heat kernel associated with the Grushin operator $\Delta_G = \Delta_x + |x|^2 \Delta_u$ on $ \R_x ^n \times \R_u ^{n'} $ for all $ n , \, n' \ge 1 $. We further establish sharp bounds for its spatial derivatives. As a by-product, precise estimates and small-time asymptotic behaviour for the Grushin heat kernel will be provided, as explicit as one can possibly hope for.
\end{abstract}

\section{Introduction}

On $\R_x^n \times \R^{n'}_u$ endowed with Lebesgue measure, the Grushin operator, initially introduced by V. V. Grushin (cf. \cite{G70}), is defined as
    $$ \Delta_G = \Delta_x + |x|^2 \Delta_u = \sum_{j = 1}^n \left( X_j^2 + \sum _{ k = 1 }^{n'} U_{ j k }^2 \right), \quad \mbox{ where } X _ j \coloneqq \frac{ \partial }{ \partial x_j}, \, U_{ j k } \coloneq x_j \, \frac{ \partial }{ \partial u_k}, $$
and $| \cdot |$ denotes the usual Euclidean norm. The operator loses ellipticity along $ \{ x = 0 \} $, but remains hypoelliptic. Moreover, despite the absence of an underlying group structure, it is a toy model of nil-manifolds and admits a natural anisotropic dilation structure. During the past few decades, there has been considerable work on the Grushin operator (and more general frameworks) regarding various problems, such as weighted Sobolev-Poincar\'e inequalities, geometric properties, Riesz transforms and spectral multipliers, etc. 

Let us denote by  $d ( \ggg) $ the Carnot–Carath\'eodory distance associated with $\Delta_G$, where $\ggg \in \R_x ^n \times \R _u ^{n'} $. Our main concern in this paper is the uniform asymptotic behaviour at infinity (i.e. $d (\ggg)^2 / h \to + \infty$) of the Grushin heat kernel $p_h (\ggg) \,( h > 0 )$, which is the fundamental solution of $\frac{\partial}{\partial h} - \Delta_G$. As corollaries, we obtain its small-time asymptotics as well as the uniform upper and lower bounds, as explicit as one can possibly hope for. In our specific context, several general results available in the literature yield useful, though non-sharp, estimates. For example, the small-time asymptotic expansions of the heat kernel strictly off the cut locus which G. Ben Arous derived in \cite{BA88} provide a preliminary answer, but the estimates are far from sharp in our setting. A. Sikora's result \cite{Sikora04} yields an upper bound for all $ \ggg \in \R_x ^n \times \R _u ^{ n '} $ and $ h > 0 $: 
\begin{equation} \label{ineq_upbd}
    p_h ( \ggg) \le C \, \left( 1 + \frac{ d (\ggg) }{ \sqrt{ h } } \right)^{-1} V _ g\left( \frac{ h }{ d (\ggg) + \sqrt{ h } } \right) ^ { - \frac12 } \, V _ {g'} \left( \frac{ h }{ d(\ggg) + \sqrt{ h } } \right) ^ { - \frac12 } e^ { - \frac{ d (\ggg) ^2 }{ 4 \, h } } ,
\end{equation}
where $C > 0$ is a universal constant, $V_g(r)$ denotes the volume of the ball of center $g = ( x , u )$ and radius $r$, and satisfies (cf. e.g. \cite[Proposition~5.1]{RS08}, also \cite[Theorem~2.3]{Franchi91} or \cite[Proposition~2.2]{FGW94})
\begin{equation} \label{12f}
    C^{-1} \le \frac{V_g (r)}{r^{ n + n ' } \left( r ^{ n' } + | x | ^{ n ' } \right)} \le C, \qquad \forall g = (x, u) \in \R_x ^n \times \R _u ^{ n '}, \  r > 0.
\end{equation}

Furthermore, \cite[Theorem~4.11]{CS08} leads to the following gradient estimate:
\begin{equation}\label{daoshu_up}
    \left| \nabla_G \, p_h (\ggg) \right| \le C \, \frac{ ( 1 + d(\ggg)^2 / h ) ^{ 3 ( n + 2 n ' ) +1 } }{ \sqrt{ h } \, V_{g'} ( \sqrt{h} ) } \, e^{ - \frac{ d (\ggg) ^2 }{ 4 \, h } } , \quad \forall \, \ggg \in \R_x^n \times \R_u ^{n'} \mbox{ and } h>0,
\end{equation}
where $\nabla_G \coloneqq \Big( ( X_j )_{1 \le j \le n} , \, ( U_{ j k } )_{1 \le j \le n , \, 1 \le k \le n'} \Big)$. In particular, since the Grushin operator is a nil-manifold, we further have the following estimates for $p_h ( \ggg)$ for any $h > 0 , \, \ggg \in \R_x^n \times \R_u ^{n'}$ and $0 < \epsilon < 4$ as a result of \cite[Th\'eor\`emes 2.12 et 2.14]{M98}: 
\begin{gather}
    p_h ( \ggg ) \ge C_\epsilon \, \frac{ \exp \left( - \frac{ d (\ggg) ^2 }{ ( 4 - \epsilon ) \, h} \right) }{  V _g ( \sqrt{h} ) ^ \frac12 \, V_{g'} ( \sqrt{ h } ) ^ \frac12 } , \label{ineq_lowerbd} \\
    \left| \nabla_G^k \, p_h (\ggg)  \right| \le C_{ \epsilon , \, k } \, h^{ - \frac{k}{ 2 } } \, \frac{ \exp \left( - \frac{ d (\ggg) ^2 }{ ( 4 + \epsilon ) \, h} \right) }{  V _g ( \sqrt{h} ) ^ \frac12 \, V_{g'} ( \sqrt{ h } ) ^ \frac12 }, \quad \mbox{ where } k = 1, 2, \ldots. \label{ineq_derivupbd}
\end{gather} 

Thereafter, significant progress was made by \cite{Li12}, where uniform asymptotics at infinity for the heat kernel, together with precise pointwise estimates as well as small-time asymptotics, were established in the case $ n \ge 3 $ and $n ' = 1 $. A subsequent work \cite{LZ19} removed the additional condition of $n \ge 3$ and showed that \cite[Theorems 1.2 and 1.3]{Li12} are valid for $n \ge 1$. 

Recall that the Grushin operator can be considered as a nil-manifold related to the  Heisenberg--Reiter group $\mathbb{H}_{n, n'}$ endowed with its canonical sub-Laplacian $\Delta_{n, n'}$; and the transference technique has often proved effective for some analytic problems. See for instance \cite{MS12} for spectral multipliers. However, this technique appears to be of limited use for precise heat kernel estimates, even in the simplest Grushin case \(n = n' = 1\). Furthermore we point out that $(\mathbb{H}_{n, 1}, \Delta_{n, 1})$ is exactly the famous Heisenberg group of dimension $2 n + 1$. For $n' \ge 2$, using \cite[Proposition~2.2]{li2021carnotcaratheodory}, it is easy to show that $(\mathbb{H}_{n, n'}, \Delta_{n, n'})$ admits the GM-property introduced therein. Hence the basic geometric properties can be readily deduced from \cite{li2021carnotcaratheodory}, such as the explicit expression for the related control distance. Nevertheless it seems to be intractable for deriving precise heat kernel bounds for $(\mathbb{H}_{n, n'}, \Delta_{n, n'})$ in the case where $n, n' \ge 2$,
as explicit as one can possibly hope for. We point out that sharp heat kernel estimates will be provided in \cite{LTY26} for the remaining case $n = 1$ and $n' \ge 2$.

Instead, our approach is based on \cite{li2021carnotcaratheodory} and \cite{LZ2025}, and remains applicable despite the absence of the group structure. It may also be viewed as a refinement and continuation of \cite{Li12}. Recall that the underlying idea of \cite{li2021carnotcaratheodory} is operator convexity. From the methodological point of view, there is no essential difference between the Grushin operator and H-type groups. This partly explains why the resulting estimates are closely parallel to those obtained for (generalized) H-type groups.

Since \eqref{ineq_upbd} and \eqref{ineq_lowerbd} already establish the uniform estimates for $p_h(\ggg)$ with $\frac{d (\ggg)^2}{h} $ bounded from above, the present paper mainly focuses on the uniform asymptotic behaviour at infinity for \(p_h (\ggg)\) in the general case $n , n' \ge 1$; see Theorems \ref{Thm_est_of_p_epsl_gtr_1} and \ref{Thm_est_of_K}. As consequences, we arrive at Theorem \ref{thm_main} and Section \ref{Sn6}, which provide, respectively, sharp uniform bounds and small-time asymptotics for \(p_h (\ggg)\). We also obtain sharp upper estimates for the spatial derivatives of $p_h (\ggg)$, stated in Theorem \ref{thm_deriv}.
   
The remainder of this paper is organized as follows. In Section \ref{Sn2}, we introduce the necessary notation and state our results regarding sharp uniform estimates for the Grushin heat kernel and its spatial derivatives. Section \ref{Sn3} treats the asymptotic behaviour for the heat kernel in the simple case. Sections \ref{Sn4} and \ref{Sn5} are devoted to the difficult case, where we establish the uniform asymptotic formula for \( p_1 (\ggg)\) and derive the sharp two-sided bounds. In Section \ref{Sn6}, we deduce the small-time asymptotics for $ p_h ( \ggg ) $ as a corollary. Finally, upper bounds for the spatial derivatives of $p_1 ( \ggg )$ are presented in Section \ref{Sn7}.

    \noindent
	\textbf{Notation.\;} With a slight abuse of notation, we will also write $|E|$ for the Lebesgue measure of a measurable set $E$, and $\cdot$ for the Euclidean inner product. $\nabla_{\theta}$ (resp. $\hess_{\theta}$) with $\theta \in \R^k$ denotes the usual gradient (resp. Hessian matrix) on $\R^k$.
  
	The symbols $C, c$ will be used throughout to denote implicit positive constants, which may vary from one line to the next. Whenever necessary, we use subscripts to specify the parameters which the constants $C$ and $c$ depend on. For non-negative functions $f_1$ and $f_2$, we adopt the notation $f_1 \sim f_2$ if there exists  $C>0$ such that $ C^{-1} f_2 \leq f_1 \leq C f_2$. Similarly, $f_1 \lesssim f_2$ (resp. $f_1 \gtrsim f_2$) if $f_1 \le C f_2$ (resp. $f_1 \ge C f_2$). Moreover, we will use the counterparts of such notation for positive semidefinite real matrices. Furthermore, for a complex-valued function $w$, by $w = O(f_1)$ we mean $|w| \le C f_1$.

    \section{Main results }\label{Sn2}

    Fix two points $g = (x , u ), \, g' = (x' , u') \in \R_x^n \times \R_u^{n'}$. We introduce the quantities
    \begin{gather}
        R^2 = |x|^2 + |x'|^2 ,\quad a = \frac{2 \, x \cdot x'}{ R^2 } \in [-1,1],\quad \xg= x + x',\quad r = |u -u'|.
    \end{gather}
    These parameters will be used throughout the paper.
     To facilitate the expression of the heat kernel, we further define 
    \begin{equation} \label{def_psi}
        \pa (\rr) \coloneqq \rr \coth \rr - a \, \frac{ \rr }{ \sinh \rr }=1- a + 2 \sum_{k=1}^\infty \frac{\rr^2 \, \left( 1 - a \, (-1)^k\right)}{\rr^2 + k^2 \,\pi^2},
    \end{equation}
 which follows from the classical expansions (cf. e.g. \cite[\S 1.421.3, \S 1.422.3]{GR15})
     \begin{equation} \label{eq_expansion}
        -\rr \cot \rr = -1 + 2\sum_{j=1}^ \infty \frac{\rr^2}{(j\pi)^2 - \rr^2 },\qquad \rr \csc \rr = 1 - 2\sum_{j=1}^\infty \frac{(-1)^{j} \, \rr^2}{(j\pi)^2-\rr^2} .
    \end{equation}
    Notice that 
    \begin{align}
       \psi_{-1}(\rr) = \rr \, \coth{\frac{\rr}{2}}. 
    \end{align}
    Then the expression of the Grushin heat kernel can be written as 
     \begin{align} \label{Eq_def_of_p}
         p_h (\ggg) & =  (4 \pi h) ^{-\frac{n}{2}-n'} \!\! \int_{\R^{n'}} \!\! \bV(\lambda) \, \exp \left\{ -\frac{1}{4h}\wtreffun(\ggg;\lambda)\right\} \, d \lambda,
     \end{align}
     where
    \begin{gather}
        \wtreffun( \ggg ; \lambda ) = R^2 \, \pa( |\lambda| ) - 2 \, i  \, ( u - u' ) \cdot \lambda,\\
        \bV( \lambda ) = \left( \frac{ |\lambda| }{ \sinh |\lambda| }\right)^{ \frac{n}{2} } = \prod_{ k \geq 1 } \left( 1 + \frac{ \lambda \cdot \lambda }{ k^2 \pi^2 } \right)^{ - \frac{n}{2} }. \label{26f}
    \end{gather} 
    In view of \eqref{eq_expansion}, \eqref{def_psi} and \eqref{26f}, we can extend the domain of $\wtreffun( \ggg ; \lambda )$ and $\bV( \lambda )$ to $\R^{n'} + i B_{\R^{n'}}(0, \pi) \subset \C^{n'}$, where $B_{\R^{n'}}(0, \rr)$ denotes the usual open ball centered at $0 \in \R^{n'}$ with radius $\rr > 0$. Then the meaning of $|\lambda| \csch{|\lambda|}$ and $|\lambda| \coth{|\lambda|}$ is obvious for suitable $\lambda \in \C^{n'}$.
    
    By virtue of the scaling property, namely
      \begin{equation} \label{eq_sp}
          p _h \Big( ( x , u ) , \, ( x' , u' ) \Big) = h^{ - \frac n 2 - n' } \, p_1 \left( \left( \frac{ x }{ \sqrt{ h } } , \frac{ u }{ h } \right) , \, \left( \frac{ x' }{ \sqrt{ h } }  , \frac{ u ' }{ h } \right) \right),
      \end{equation}
    it suffices to consider the heat kernel at time $1$, which we denote by $p(\ggg)$ for simplicity.

    We next derive the explicit expression for the Carnot-Carath\'eodory distance associated with $\Delta_G$. Set in the sequel,
    \begin{align}
     \ma( \rr )  := - \frac{ d }{ d \rr } \, \pa( i \, \rr ) =  4 \, \rr  \sum_{k=1}^\infty \frac{k^2 \, \pi^2 \, \left( 1 - a \, (-1)^k\right)}{(k^2 \,\pi^2 - \rr^2)^2},
    \end{align}
    which is obviously a monotonically increasing diffeomorphism from $(-\pi, \ \pi)$ onto $\R$ provided $-1 < a \le 1$, while from $(-\pi, \ \pi)$ onto $( - \frac{ \pi }{ 2 } ,  \ \frac{ \pi }{ 2 })$ for $a = -1$ by the fact that 
    \[
    \mu_{-1}(\rr) = \frac{\rr - \sin{\rr} }{2 \, \sin^2{\frac{\rr}{2}}}.
    \]
    Let $\ma^{-1}$ denote its inverse. It follows from \eqref{def_psi} that $\pa( i \, \rr )$ is operator convex on $(-\pi, \, \pi)$ in the sense of \cite{bhatia2013matrix}, then the following result can be readily deduced by applying the powerful method introduced in \cite{li2021carnotcaratheodory}:

    \begin{theorem}[\cite{li2021carnotcaratheodory}] We have:
    \begin{align}\label{Eq_expression_of_d}
       & d(\ggg)^2 = \sup_{\lambda \in B_{\R^{n'}}(0, \pi)}  \wtreffun( \ggg ; i \, \lambda ) \nonumber \\
       &= \left\{\begin{aligned}
            & 2 \pi r, \qquad\qquad\qquad\qquad\qquad\qquad\qquad\qquad\qquad\quad x = -x' \text{ and } 2r \geq \pi |x|^2 ,\\
            & \left( \frac{ | \theta | }{ \sin | \theta | } \right) ^ 2  R^2 ( 1 - a \cos |\theta| ) \text{ with } \theta = \ma^{-1} \left( \frac{ 2 r }{ R ^ 2 }\right) \frac{ ( u - u ' ) }{ r }, \quad \text{otherwise}.
        \end{aligned}\right.
    \end{align}
    \end{theorem}

    In the special case where $n' = 1$, the control distance has been obtained in \cite{Paulat07ax}, by using a classical method. However, the more practical and robust approach adopted in \cite{li2021carnotcaratheodory} is based on the well-known Varadhan formulas and the operator convexity. 

    In what follows, unless otherwise specified, we always suppose that:
    \begin{align}
     \ggg \in \R_x^n \times \R _ u ^{n'} \  \mbox{with} \ \xg \ne 0 \ \mbox{and} \ \varepsilon := \pi - |\theta| > 0.  \tag{AN}
    \end{align}

    Our ultimate goal is to establish uniform asymptotic expansions for the Grushin heat kernel at infinity. To avoid introducing the rather involved notation needed for the full expansions at this point, we first state the resulting uniform sharp bounds for $p(g,g')$. The corresponding uniform asymptotic formulas will be derived in Sections \ref{Sn3} and \ref{Sn5}.

    \begin{theorem} \label{thm_main}
        Let $ \beta \coloneq \sqrt{ \frac{ 1 + a }{ 1 + a + \varepsilon^2 }}$. Then we have uniformly, under the condition (AN), that: 
        \begin{equation*}
            p (\ggg) \sim  e^{- \frac{ d(\ggg) ^2 }{ 4 } } \frac{ ( 1 + d(\ggg) ) ^ \frac n 2  \, {( 1 + d(\ggg) + R)}^{ - n' }  }{ 1 + \beta \sqrt{ \varepsilon} \, d(\ggg) + \varepsilon \, d(\ggg)} \left( \frac{ 1 + \beta  \sqrt{ d(\ggg) } }{ 1 + \beta \sqrt{ \varepsilon } \, d(\ggg) + \sqrt{\varepsilon \, d(\ggg)} } \right)^{n - 2}.
        \end{equation*}
    \end{theorem}

    \begin{remark}
        When $n'=1$, our uniform estimates recover the corresponding results established in \cite{Li12} for $n \ge 3$ and \cite[Theorem 6 and Remark 10]{LZ19} for $n = 1, 2$.
    \end{remark}

    As an analogue of the classical Euclidean setting, we obtain the following sharp bounds for spatial derivatives of \(p_h(\ggg)\). 
    
    \begin{theorem} \label{thm_deriv}
        Let $\nabla_G = \Big( (X_j)_{1 \le j \le n}, \, ( U_{jk}) _{1 \le j \le n, \, 1 \le k \le n' } \Big)$ denote the horizontal gradient. It holds, for any $g , \, g' \in \R_x^n \times \R_u ^{n'}$ and $h > 0$, that
        \begin{equation}
   \left| \nabla_G^l \, p_h(\ggg) \right| \lesssim_l \, h^{- \frac l 2} \left( 1 + \frac{d (\ggg)}{ \sqrt{h} } \right)^l \, p_h (\ggg), \quad l = 1, 2, 3, \ldots.
        \end{equation}
    \end{theorem}

    \begin{remark} 
     In contrast to the classical Euclidean setting and \cite[(6.1) and (8.2)]{LZ2025}, the strongest gradient estimate of type $\left| \nabla_G \, \ln{p_h(\ggg)} \right| \lesssim d(\ggg)/h$ fails because 
     $\frac{\partial}{\partial x_1} p((x, u),(x, u)) \neq 0$ whenever  $x = (x_1, \ldots, x_n)$ with $x_1 \neq 0$. 
    \end{remark}

    \section{\texorpdfstring{Uniform asymptotics for \(p\) in the simple regime: $ \varepsilon \ge \delta_0^* > 0$}{}}\label{Sn3}

    Let us begin with uniform asymptotics of \(p(\ggg)\) under the assumption that $d(\ggg) \longrightarrow +\infty$ with $\varepsilon = \pi - |\theta| \ge \delta_0^* > 0$, where $0 < \delta_0^* \le \frac{\pi}{8}$. In such a case, it follows from \eqref{Eq_expression_of_d} that $d(\ggg)^2 \lesssim_{\delta_0^*} R^2$. The following more general asymptotics for \(p(\ggg)\) hold:
    
     \begin{theorem}\label{Thm_est_of_p_epsl_gtr_1} 
        Let $0 < \delta_0^* \le \frac \pi 8$. Then there exists a constant $C(\delta_0^*) > 0$ such that:
        \begin{equation} \label{Eq_asym_of_p_epsl_gtr_1}
        p(\ggg) = \frac{ (4 \pi ) ^{- \frac{n}{2} - n'} (8 \pi ) ^{\frac{n'}{2}} \, \bV ( i \theta )  }{ \sqrt{ \det ( - \hess_\theta \, \reffun ( \ggg ; \theta ) )} } \, e^{- \frac{ d(\ggg)^2 }{ 4 }} \, ( 1 + \O_{\delta_0^*} ( R ^{-\frac 14} ) ), 
    \end{equation}
    whenever $R \ge C(\delta_0^*)$ with $\varepsilon \ge \delta_0^*$; as a consequence, in such a situation, we have:
    \begin{equation} \label{Eq_est_of_p_epsl_gtr_1}
            p  (\ggg) \sim_{\delta_0^*}
            R^{-n'} e^{-\frac{d(\ggg)^2}{4}} \sim e^{-\frac{d(\ggg)^2}{4}} \, \min\Big\{ \frac{1}{V_g(1)}, \frac{1}{V_{g'}(1)} \Big\} \quad \mbox{(cf. \eqref{12f})}. 
    \end{equation}   
    \end{theorem}

    In particular, if we disregard the order of the remainder, it extends \cite[(1.19)]{Li12} in the case where $n' = 1$. Theorem \ref{Thm_est_of_p_epsl_gtr_1} follows readily from the method introduced in 
    \cite{li2021carnotcaratheodory}. Although the Grushin operator $\Delta_G$ does not admit any group structure, it can still be considered as a special case of \cite[Theorem~4.1]{LZ2025}. Indeed, for the simple case, the proof in \cite[pp. 13-17]{LZ2025} can be made more elementary by means of \eqref{def_psi} and \eqref{26f}. Details are left to the interested reader.

\section{\texorpdfstring{Sharp estimates for $p$ in the difficult case: 
$0 < \varepsilon \le \delta_0 := \frac{\pi}{8}$}{}} \label{Sn4}

    Define
    \begin{align} \label{ygf}
        \delta_0 := \frac{\pi}{8}, \qquad \yg = \left(1 - \frac{|\theta|^2}{\pi^2}\right)^{-\frac{1}{2}} \xg, \quad \mbox{then $|\yg|^2 \sim \frac{R^2 \, (1 + a)}{\varepsilon}$.}
    \end{align}
    In the case where \( 0 < \varepsilon \le \delta_0 \), we have the following sharp uniform bounds for \(p (\ggg) \):

    \begin{theorem} \label{Thm_est_of_p_epsl_ll_1}
        Let $\ggg \in \R_x^n \times \R^{n'}_u$ satisfy \(0 < \varepsilon \le \delta_0\). Then there exists a sufficiently large constant $\BR > 0$ so that, whenever \(d (\ggg) \ge  \BR\), 

        (i) if \(1 + a \ge \varepsilon^2\), then
        \begin{equation} \label{Eq_1+a_gtr_epsl}
            p (\ggg) \sim e^{-\frac{d(\ggg)^2}{ 4 }} \, \frac{ d (\ggg) ^{ n - n' - 1 } }{\left(1 + \sqrt{\varepsilon } \,d (\ggg) \right)^{n-1}};
        \end{equation}

        (ii) if \(1 + a \le \varepsilon^2\), then
        \begin{equation} \label{Eq_1+a_small}
            p (\ggg) \sim e^{- \frac{d (\ggg) ^2}{4}} \, \varepsilon^{-\frac{n}{2}} \, d(\ggg)^{-n'} \frac{ \varepsilon \, d (\ggg) }{1 + |\yg| + \varepsilon \, d (\ggg) }\left(\frac{|\yg|^2 + \varepsilon \, d(\ggg) }{ 1 + |\yg|^2 + \varepsilon \, d (\ggg) }\right)^{\frac{n}{2} - 1} .
        \end{equation}
    \end{theorem}
    \begin{remark}
    The estimate \eqref{Eq_1+a_gtr_epsl} can also be expressed in terms of \eqref{Eq_1+a_small}. In particular, if we put \( g = ( 0 , 0 ) \), then \(a = 0 \) and \eqref{Eq_1+a_gtr_epsl} encompasses the corresponding result for H-type groups.
    \end{remark}

    As a direct consequence of \eqref{Eq_expression_of_d}, we have the following equivalent expression for $d^2$:
    
    \begin{lemma}
    We have:
    \begin{equation}\label{Eq_esti_of_d}
                d(\ggg)^2 \sim R^2 \, \max\Big\{ \frac{1+a}{\varepsilon^2},  \ 1 \Big\} \sim R^2 \, \frac{ 1 + a + \varepsilon^2 }{ \varepsilon^2 }, \qquad \mbox{whenever $\ggg$ satisfy \(\varepsilon \le \delta_0\).}
            \end{equation}
    \end{lemma}

    The proof of Theorem \ref{Thm_est_of_p_epsl_ll_1} is postponed to the section \ref{Sn5}. Combining this fact with Theorem \ref{Thm_est_of_p_epsl_gtr_1} as well as the classical Li-Yau estimates (cf. \eqref{ineq_upbd} and \eqref{ineq_lowerbd}), Theorem \ref{thm_main} can be explained easily by \eqref{Eq_esti_of_d}. In particular, the following simple observation will be used: it holds that $|\yg| \sim \beta \, \sqrt{\varepsilon} \, d(\ggg)$ 
    provided $ \varepsilon \le \delta_0$ 
    (cf. \eqref{ygf}).

\subsection{Preliminaries}

When $\varepsilon \le \delta_0$, both the phase and the amplitude become singular as $|\theta| \to \pi$. To isolate the singular part, we rewrite the heat kernel using the Fourier transform of Gaussian functions, following the strategy of \cite{LZ2025}. To this end, we introduce 

    \begin{gather}
        \wtff( \ra ; \lambda ) \coloneqq R^2 \left( \pa( | \lambda | )- \frac{ 2 \,  \lambda \cdot \lambda }{  (\lambda \cdot \lambda) + \pi^2 } \, (1 + a)\right), \label{Eq_def_of_wtff} \\
        \bVt( \lambda ) \coloneqq \left( 1 + \frac{  \lambda \cdot \lambda }{ \pi ^ 2 } \right)^{ \frac n 2 } \cdot \bV( \lambda ) = \prod_{ k \geq 2} \left( 1 + \frac{  \lambda \cdot \lambda }{ k ^ 2 \, \pi ^ 2 } \right) ^ { - \frac n 2 }. \label{Eq_def_of_V2}
    \end{gather} 
    We further define $\ff( \ra ; \lambda ) := \wtff( \ra ; i \lambda )$. 
      It follows from \eqref{def_psi} that:
    \begin{equation}\label{Eq_definition_of_f(Ra;r)}
        \ff(\ra;\lambda) = R^2 \, (1 -a ) - 2 \, R^2 \, \sum_{k\geq 2} \frac{\lambda \cdot \lambda}{k^2\,\pi^2 - \lambda \cdot \lambda } \left( 1 - a \, (-1)^k \right),
    \end{equation}
    and thus $\wtff(\ra;\lambda)$ (also $\bVt( \lambda )$) can be extended to a holomorphic function on 
    \(\R^{n'} + i B_{\R^{n'}}(0,2\pi)\). 
     As a consequence, a direct computation yields the following lemma.
   \begin{lemma} \label{lemma_hess}
       It holds that: 
       \begin{gather}
           -\hess _\vt \, \ff(\ra;\vt) \sim  R^2 \, \I_{n'}, \qquad \forall \, \vt \in B_{\R^{n'}} \left(0,  \frac{15}{8} \pi \right), \  R \ge 0, \ -1 \le a \le 1; \label{48fn} \\
           \Re \Big( \ff(\ra;\vt - i \lambda) - \ff(\ra;\vt) \Big) \gtrsim \frac{|\lambda|^2}{ 1 + |\lambda| } \, R^2, \quad \forall \, \vt \in B_{\R^{n'}} (0,  2 \pi ), \  R \ge 0, \ -1 \le a \le 1; \label{410f} \\
           |\bVt(\lambda + i\vt) | \leq \bVt(\lambda) \cdot \bVt(i\vt), \qquad \forall \vt \in B_{\R^{n'}} (0, 2 \pi ), \  \lambda \in \R^{n'}. \label{49f}
       \end{gather}
   \end{lemma}

Recall the assumption (AN), with $\theta$ defined as in \eqref{Eq_expression_of_d}. Denote in the sequel for $\xi \in \R^n$ and $\lambda \in \R^{n'}$ that:
        \begin{gather}
            \Wg  \coloneq \frac{|\yg|^2 - |\xi|^2}{\pi^2 - |\theta|^2} \, \theta,\quad \Ag  \coloneq \frac{|\xi|^2}{\pi ^2 - |\theta|^2} \, \I_{n'}- \frac{1}{4} \, \hess_\theta \, \ff(\ra;\theta), \label{411fn}\\
            \Sg(\lambda) \coloneq \ff( \ra ; \theta - i \lambda) - \ff(\ra;\theta) +  i \, \lambda \cdot \nabla_\theta \, \ff ( \ra ; \theta ) + \frac{1}{2} \, \lambda \t \, \hess_\theta \, \ff( \ra ; \theta ) \, \lambda, \label{Eq_def_of_S()} \\
         \cpi \coloneq \frac{1}{(2\pi)^{\frac{n}{2}} \, (4\pi)^{n' + \frac{n}{2}}}, \quad   \Sg(\xi;\lambda) \coloneq \Sg(\lambda) + 2 \, \lambda\t \, \Ag \, \lambda - 4 \, i\,\lambda \cdot \Wg . \label{413nf}
        \end{gather}

The following Laplace-type integral expression for \( p \) can be considered as a special instance of \cite[Theorem~5.1]{LZ2025}:

    \begin{proposition}\label{Thm_transform}
        It holds that: 
        \begin{equation}\label{Eq_transformed_2_of_p}
            p(\ggg) = \cpi \, e^{-\frac{d(\ggg)^2}{4}} \left(1 - \frac{|\theta|^2}{\pi^2}\right)^{- \frac{n}{2} } \int_{\R^n} e^{ -\frac{1}{2} |\xi - \yg|^2 } \, \F(\ggg;\xi) \d \xi,
        \end{equation}
        where
         \begin{align}\label{Eq_def_of_F} 
             \F(\ggg;\xi) \coloneq \int_{\R^{n'}}\bVt(\lambda + i \theta) \, \exp \left( - \frac{1}{4} \, \Sg(\xi;\lambda)\right) \d \lambda .
        \end{align} 
    \end{proposition}
    \begin{proof}
    For the reader's convenience, we provide a concise proof. First, notice that:
     \begin{align*}
                \left(1 + \frac{|\lambda|^2}{\pi^2}\right)^{-\frac{n}{2}} = (2\pi)^{-\frac{n}{2}} \exp\left( - \frac{ \pi^2 \, |\xg|^2 }{ 2 \, ( |\lambda|^2 + \pi^2 ) } \right)  \int_{\R^n}  \exp \left(-\frac{1}{2}\left(1 +\frac{|\lambda|^2}{\pi^2} \right)|s|^2 + \xg \cdot s\right)\d s .
        \end{align*}
    
    Substituting this into \eqref{Eq_def_of_p}, together with the definition of \(\bVt\) and \(\wtff\) (cf. \eqref{Eq_def_of_V2} and \eqref{Eq_def_of_wtff}), 
    it follows from $|\xg|^2 = R^2 \, (1 + a)$ that:
        \begin{align} \label{Eq_transformed_1_of_p}
            p(\ggg) &= \cpi \int_{\R^{n'}} \bVt(\lambda) \, e^{-\frac{1}{4} \left( \wtff( \ra ; \lambda ) - 2 \, i \, (u - u')\cdot \lambda \right)} \left(\int_{\R^n} \!\! e^{ -\frac{1}{2} |s-\xg|^2-\frac{ (\lambda \cdot \lambda) \, |s|^2}{ 2 \, \pi^2 }} \d s\right) \d \lambda \nonumber \\
            &= \cpi \int_{\R^n} \!\! e^{ -\frac{1}{2} |s-\xg|^2}  \d s \int_{\R^{n'}} \bVt(\lambda) \, e^{-\frac{1}{4} \left( \wtff( \ra ; \lambda ) - 2 \, i \, (u - u')\cdot \lambda + 2 \frac{ (\lambda \cdot \lambda) \, |s|^2}{\pi^2 }\right)} \d \lambda.
        \end{align}
        Let $\mathbf{E}$ denote the inner integral. Observe that its integrand is holomorphic in \(\R^{n'} + i B_{\R^{n'}} (0, 2 \pi)  \subset \mathbb{C}^{n'}\). From \eqref{410f} and \eqref{49f}, we can deform the contour from \(\R^{n'}\) to \(\R^{n'} + i\theta\) to get 
        \begin{align*}
        \mathbf{E} = \int_{\R^{n'}} \bVt(\lambda + i \theta) \, \exp\left\{ -\frac{1}{4} \left(\ff(\ra;\theta-i\lambda) - 2 \, i \, (u-u') \cdot(\lambda + i\theta)+ \frac{ 2 \, \langle \lambda + i \theta \rangle^2  }{ \pi^2 } |s|^2 \right) \right\} \d \lambda,
        \end{align*}
    where $\langle z = (z_1, \ldots, z_{n'}) \rangle^2 := \sum_{j = 1}^{n'} z_j \cdot z_j$ provided $z \in \C^{n'}$.

    Next, recall that $\yg = \left(1 - \frac{|\theta|^2}{\pi^2}\right)^{-1/2} \xg$. By the fact that $d(\ggg)^2 = \reffun(\ggg;\theta)$ and $\nabla_\theta \, \reffun \, (\ggg;\theta) = 0$, we obtain that:
        \begin{gather}
        d(\ggg)^2 = \ff(\ra; \theta) - 2 \, |\yg|^2 + 2 \, |\xg|^2  + 2 \, (u-u')\cdot \theta, \label{eq_of_d^2} \\
            \nabla_\theta \, \ff(\ra;\theta) + 2 \, (u-u') =  \frac{4 \, \pi^2 \, |\xg|^2 }{\left(\pi^2 - |\theta|^2\right)^2} \, \theta = 4 \, \frac{|\yg|^2}{\pi^2 - |\theta|^2} \, \theta. \label{eq_of_nable_f} 
        \end{gather}
        Therefore, a direct calculation shows that the phase of $\mathbf{E}$ equals to 
    \[
    -\frac{1}{4} \left\{ d(\ggg)^2 + \Sg\left((1- \frac{ |\theta| ^2}{ \pi ^2})^{\frac{1}{2}} \, s; \lambda\right) + 2 \, |\yg|^2 - 2 \, |\xg|^2 - 2 \,  |s|^2 + 2 \, \frac{\pi^2 - |\theta|^2}{\pi^2} |s|^2 \right\}.
    \]
    This together with \eqref{Eq_transformed_1_of_p}, the change of variables \(s = \left(1- \frac{ |\theta| ^2}{ \pi ^2}\right)^{-\frac{1}{2}} \xi \) implies \eqref{Eq_transformed_2_of_p}.
   \end{proof}

\subsection{Estimates of $\F(\ggg;\xi)$} 
   
 To derive uniform heat kernel asymptotics at infinity for $(\ggg)$ satisfying $0 < \varepsilon \le \delta_0$, we first need to  investigate the properties of $\F(\ggg;\xi)$. The following asymptotic behaviour of $\F(\ggg;\xi)$ is indeed a specialization of \cite[Theorem~5.2]{LZ2025} to the present setting,
which is based on a perturbation argument closely related to the method of stationary phase.

    \begin{theorem} \label{thm_jde52} 
        Define 
        \begin{align} \label{EgEf}
          \Lgg \coloneq R^2 + \frac{|\xi|^2}{\varepsilon}, \qquad  \EgE := \Wg \t \, \Ag^{-1} \, \Wg, \qquad \forall \, \xi \in \R^n.   
        \end{align}
        Let $\vpo > 0$. There exists a sufficiently large constant $C(\vpo) > 0 $ such that
        \begin{equation}
        \F (\ggg , \xi) = (2\pi)^{\frac{n'}{2}} \, \bVt(i\theta)\,  \frac{\exp \left( - \frac{1}{2} \EgE \right)}{ \sqrt{ \det \Ag  } } \left(1 + \O_{\vpo}(\Lgg^{-1/8})\right),
        \end{equation}
        provided $\Lgg \ge C (\vpo)$ and $\EgE \le \vpo \, \Lgg ^{\frac 1 4}$.
    \end{theorem}

    \begin{proof}
        For the sake of completeness, we sketch the proof below.

First of all, it follows from \eqref{411fn} and \eqref{48fn} that 
\begin{align} \label{421nf} 
  \Ag \sim \Lgg \, \I_{n'}, \qquad \forall \, \xi \in \R^n.  
\end{align}
Set \[ \rg  = \Ag ^{-1} \, \Wg. \] With the restriction on $\xi$, it holds that $| \rg | \sim \sqrt{\EgE/\Lgg} = \O _{\vpo} (\Lgg ^{ - \frac 3 8 }) \ll 1$. From \eqref{410f}-\eqref{49f}, we may lift the contour in \eqref{Eq_def_of_F} by \(i \rg\), and get
\[
\F(\ggg;\xi)  = \int_{\R^{n'}}\bVt(\lambda + i (\theta + \rg )) \exp \left(-\frac{1}{4} \, \Sg(\xi;\lambda + i \rg)\right) \d \lambda. 
\]
A simple calculation gives that:
\begin{align}
\Sg(\xi;\lambda + i\rg ) = \Sg(\lambda + i \rg) + 2 \, \lambda\t \, \Ag \, \lambda + 2 \, \EgE .
\end{align}
Then we can write
\begin{align}
e^{\frac{\EgE}{2}} \F(\ggg;\xi)  =  \int_{\R^{n'}} \! \bVt(\lambda + i (\theta + \rg )) \,  e^{-\frac{1}{4} \Sg(\lambda + i \rg)}  \, e^{-\frac{1}{2} \lambda\t \, \Ag \, \lambda} \, d\lambda.
\end{align}

Next, we divide $\R^{n'}$ into $\DD_1 := \{ | \lambda | < \Lgg ^{ - \frac 3 8 } \}$ and $\DD_2 := \{ | \lambda | \ge \Lgg ^{ - \frac 3 8 }\}$, and denote the above integral taken over $\DD_j$ by $\widetilde{\F}_j(\ggg; \xi)$ ($j = 1, 2$). Let us start from the estimate of $\widetilde{\F}_1(\ggg;\xi)$. On \(\{ | \lambda | < \Lgg ^{ - \frac 3 8 } \}\), it follows from Taylor's formulas that:
       \[ \Sg(\lambda + i \rg)  = R^2 \, \O(|\lambda + i \rg|^3) = \O _{\vpo} \left( \Lgg^{-1/8} \right).\]
       Together with the observation that
       \[\bVt(\lambda + i(\theta + \rg)) = \bVt(i \theta) \left(1 + \O_{\vpo} ( \Lgg ^{-3/8})\right),\]
the standard Laplace method implies that     
        \begin{align}
        \widetilde{\F}_1(\ggg;\xi) & = \bVt(i\theta)  \int_{ \{|\lambda| < \Lgg^{ - 3/8 } \} } \exp \left(-\frac{1}{2} \, \lambda\t \, \Ag \, \lambda \right) \left(1 + \O_{\vpo} ( \Lgg ^{- \frac 1 8})\right) \d \lambda \nonumber \\
        & = \frac{(2\pi)^{n'} \, \bVt(i\theta)}{\sqrt{\det\Ag }} \, \left(1 + \O_{\vpo} ( \Lgg ^{-\frac 1 8 })\right) \ (\sim \Lgg ^{-\frac{n'}{2} }).
       \end{align}

We now estimate the remainder term \(\widetilde{\F}_2(\ggg;\xi)\). A direct calculation shows that:  
       \begin{align} \label{424nf}
          & 2 \lambda\t \, \Ag \, \lambda + \Re\Big( \Sg(\lambda + i \rg) \Big) \nonumber \\
          &= \, \Re \Big( \ff(\ra \, ;\theta + \rg - i \lambda) - \ff(\ra \, ;\theta + \rg) \Big)  + 2 \frac{|\xi|^2}{\pi^2 - |\theta|^2} |\lambda|^2  + \Sg(i \rg) \nonumber \\
          &\gtrsim  R^2 \frac{|\lambda|^2}{1 + |\lambda|} + \frac{|\xi|^2}{\varepsilon} |\lambda|^2 +  \Sg(i \rg) \nonumber \\
          &\gtrsim \Lgg^{ \frac 1 4 }, \qquad \forall \, |\lambda | \ge \Lgg^{ - \frac 3 8 },
       \end{align}
where we have used \eqref{410f} in the first ``$\gtrsim$'', and the fact that $|\Sg(i \rg)| \ll 1$ by the second-order Taylor expansion in the last inequality. Then it follows from \eqref{49f} that $|\widetilde{\F}_2(\ggg;\xi)| \lesssim \exp\{-c \Lgg^{ \frac 1 4 }\}$.
Combining the preceding estimates completes the proof of Theorem \ref{thm_jde52}. 
    \end{proof}

The following upper bounds of $\F(\ggg;\xi)$ can be considered as a simple instance of \cite[Proposition~5.1]{LZ2025}:

    \begin{proposition}\label{Prop_esti_of_F_rgn_others} 
        There exists a constant \(c > 0\) such that
        \begin{equation}
            |\F(\ggg;\xi)| \lesssim\exp \left\{-c \,  |\Wg | \, ( \Lgg  + 1 )^{-\frac{1}{2}}\right\}, \qquad \forall \, \xi \in \R^n.
        \end{equation}
    \end{proposition}

    \begin{proof}
    Let 
    \begin{align}\label{Eq_def_of_gx}
       \gx:=\frac{ c }{ \sqrt{ R^2 + |\xi|^2 / \varepsilon + 1 } } \frac{ \Wg }{ | \Wg | },
    \end{align}
    where $c>0$ is a sufficiently small constant to be determined. Then we may shift the contour to obtain 
    \begin{align} \label{426IEF}
       \F( \ggg ; \xi ) = \int_{ \R^{ n' } } \bVt( \lambda + i \theta + i \gx ) \exp \left( - \frac{1}{4} \, \Sg ( \xi ; \lambda + i \gx ) \right) \d \lambda .
    \end{align}
 Repeating the argument in the estimation of  \(\widetilde{\F}_2(\ggg;\xi)\) with slight modifications, we have
\begin{align*}
\Re \, \Sg ( \xi ; \lambda + i \gx ) &\ge 4 \, \gx \cdot \Wg + \ff(\ra \, ;\theta +\gx) -\ff(\ra \, ;\theta ) \\
& \quad - \gx \cdot \nabla_\theta \, \ff(\ra \, ;\theta) - 2 \, \frac{|\xi|^2}{\pi^2 - |\theta|^2} \, |\gx|^2 \\
&\ge 4 \, \gx \cdot \Wg + \O(1), \qquad \forall \, \xi \in \R^n,
\end{align*}
provided that the constant $c > 0$ as in \eqref{Eq_def_of_gx} is small enough. In other words,
\[
\exp \left( - \frac{1}{4} \, \Sg ( \xi ; \lambda + i \gx ) \right) \lesssim \exp(-\gx\cdot\Wg).
\]
Combining this with \eqref{426IEF} and \eqref{49f} yields the desired upper bounds.
\end{proof}

\section{\texorpdfstring{Uniform asymptotics for $p$ in the difficult case: 
$0 < \varepsilon \le \delta_0$}{}} \label{Sn5}

Recall the Laplace-type integral expression for $p$ (cf. \eqref{Eq_transformed_2_of_p}), the notation $\yg$ (cf. \eqref{ygf}), $\Wg$ and $\Ag$ defined by \eqref{411fn}, as well as $\EgE$ in \eqref{EgEf}. With Theorem \ref{thm_jde52} and Proposition \ref{Prop_esti_of_F_rgn_others} in hand, another main theorem can be stated as follows:

\begin{theorem}\label{Thm_est_of_K}
Let $\tcpi := (2\pi)^{\frac{n'}{2}} \, \cpi$. There exists a constant $C > 0$ such that we have
        \begin{align} \label{51main}
            p(\ggg) =  \tcpi \, \bV(i\theta) \, e^{ - \frac{ d(\ggg)^2 }{ 4 }} \int_{\R^n}
            \!\! \text{\small $\frac{\exp\left(-\frac{1}{2}|\yg-\xi|^2 - \frac{1}{2} \EgE \right)}{\sqrt{\det \Ag}} $} \d \xi  \left(1+\O(d(\ggg)^{- \frac 1 4 })\right),
        \end{align}
whenever $\ggg \in \R_x^{n} \times \R_u^{n'}$ satisfy the assumption (AN) with $d(\ggg) \ge C$ and $0 < \varepsilon \le \delta_0$.
\end{theorem}        

\begin{proof}
We distinguish two cases \(1 + a \le \varepsilon^2\) and \(1 + a \ge \varepsilon^2\).

{\bf Case 1. \(1 + a \le \varepsilon^2\).} In this case, under our assumption, it follows from \eqref{Eq_esti_of_d} that $d(\ggg) \sim R$. We divide $\R^n$ into $\adr_0 = \left\{ |\xi|< 10 \sqrt{\varepsilon} R, \ \EgE \le \sqrt{R} \right\}$, $\adr_2 = \left\{ |\xi|  \ge  
10 \sqrt{\varepsilon} R \right\}$ and $\adr_1 = \left\{ |\xi|< 10 \sqrt{\varepsilon} R, \ \EgE > \sqrt{R} \right\}$. 
In what follows, with $R \gg 1$, we set for $0\leq i \leq 2$
\begin{align} 
\jj _i := \int_{\adr_i} \!\! e^{-\frac{1}{2}|\xi - \yg|^2 } \, \F(\ggg;\xi) \d \xi,   \ \    \wtjj_i =  ( 2 \pi )^{ \frac{ n' }{ 2 } } \, \bVt ( i \theta ) \!  \int_{\adr_i} \!\! \frac{\exp\left(-\frac{1}{2} \, |\yg-\xi|^2 - \frac{1}{2} \, \EgE \right)}{\sqrt{\det \Ag}} \d \xi.
\end{align}
First, applying Theorem \ref{thm_jde52} gives that 
\begin{align} 
    \jj_0 = \wtjj_0 \, \Big(1 + \O(d(\ggg)^{-\frac{1}{4}}) \Big).
\end{align}  
Next, since 
\begin{align} \label{54fn}
|\yg| = \left(1 - \frac{|\theta|^2}{\pi^2}\right)^{-\frac{1}{2}} |\xg| = \left(1 - \frac{|\theta|^2}{\pi^2}\right)^{-\frac{1}{2}} \sqrt{R^2 \, (1 + a)} \le 2 R \, \sqrt{\varepsilon}, 
\end{align}
we have $\mathcal{Q}_{g,g'} \coloneq \left\{\xi ; \, |\xi-\yg|<  \sqrt{\varepsilon}\right\} \subset \adr_0$, on which $\EgE \lesssim 1$ and $\det \Ag \sim d(\ggg)^{2 n'}$. As a result, 
        \begin{align} \label{55fn}
            \wtjj_0 \gtrsim d(\ggg)^{-n'} \, \varepsilon^{\frac{n}{2}}.
        \end{align}  
Moreover, Proposition \ref{Prop_esti_of_F_rgn_others}, together with \eqref{55fn} and the fact that $|\Wg| \sim R \, \EgE$ provided $\xi \in \adr_1$, implies that 
\begin{align}
| \jj_1 | \lesssim e^{ - c' \, \sqrt{R} } \left| \left\{ \xi \in \R^n ; \, |\xi| < 10 \, \sqrt{\varepsilon} \, R \right\} \right| \lesssim d(\ggg)^{-\frac{1}{4}} \, \wtjj_0, \qquad \wtjj_1\lesssim d(\ggg)^{-\frac{1}{4}} \, \wtjj_0.
\end{align}
Finally, 
notice that (cf. \eqref{421nf}, \eqref{54fn} and \eqref{411fn})
\begin{align}
    \Ag  \sim \frac{|\xi|^2}{\varepsilon} \, \I_{n'}, \qquad |\Wg| \sim \frac{|\xi|^2}{\varepsilon} \sim \EgE \sim \Lgg , \qquad \forall \, \xi \in \adr_2.
\end{align}
Then we get that
        \begin{align}
            \wtjj_2 \le \sup_{\xi \in \adr_2} \frac{\exp \left( - \frac{1}{4} \, \EgE \right)}{\sqrt{\det \Ag}} \int_{\adr_2} \exp \left( - \frac{1}{4} \, \EgE \right) \d \xi \lesssim \varepsilon^{\frac{n}{2}} \, e^{-c R} \lesssim d(\ggg)^{-\frac{1}{4}} \, \wtjj_0.
        \end{align}
        Similarly, it follows from Proposition \ref{Prop_esti_of_F_rgn_others} that $|\jj_2|\lesssim d(\ggg)^{-1/4} \, \wtjj_0$. Combining the above estimates, we conclude the desired result in the case where $1 + a \le \varepsilon^2$.

{\bf Case 2. $1 + a \ge \varepsilon^2$.} In this case, instead of the above decomposition of $\R^n$, we will use 
\begin{gather*}
        \rgn_0 = \left\{ \xi \in \R^n ; \, C_1 \, d (\ggg) <{ |\xi| \over  \sqrt{\varepsilon} }< C_2 \, d (\ggg) , \, {\Big| |\xi| - |\yg|\Big| } < \sqrt{\varepsilon} \, d(\ggg) ^{\frac{1}{4}} \right\}, \\
        \rgn_1 = \left\{ \xi \in \R^n ; \, C_1 \, d (\ggg) <{ |\xi| \over  \sqrt{\varepsilon} }< C_2 \, d (\ggg) \right\} \setminus \rgn_0 ,\\
        \rgn_2 = \left\{ \xi \in \R^n ; \, {|\xi| \over \sqrt{\varepsilon}} \leq C_1 \, d (\ggg) \right\} ,\qquad 
        \rgn_3 = \left\{ \xi \in \R^n ; \, {|\xi| \over \sqrt{\varepsilon}} \geq C_2 \, d(\ggg)  \right\},
    \end{gather*}
with suitable constants $C_1, C_2 > 0$. Note that the choice of \(C_1, C_2\) depends on the implicit constants in \eqref{Eq_esti_of_d}, and the main contribution comes from $\rgn_0$. The proof is analogous to that of Case 1; alternatively, one may also refer to \cite[Sect.~5]{LZ2025}. We therefore omit the details. 
\end{proof}

\subsection{\texorpdfstring{Proof of Theorem \ref{Thm_est_of_p_epsl_ll_1}}{}} \label{Sn51}

When it comes to the relatively less difficult problem of uniform lower and upper bounds for the Grushin heat kernel, the integral in \eqref{51main} can be simplified. More precisely, from the proof of Theorem \ref{Thm_est_of_K}, it holds on the main contribution domain that:
\[
\det \Ag \sim d(\ggg)^{2 n'}, \qquad \EgE \sim \frac{|\Wg|^2}{d(\ggg)^2} \sim \frac{(|\yg|^2 - |\xi|^2)^2}{\varepsilon^2 \, d(\ggg)^2};
\]
and we can get:

\begin{corollary}
    There exist constants $c_1, c_2, C > 0$ such that:
    \begin{align} \label{59fn}
     \int_{\R^n} \!\! e^{-\frac{1}{2}|\yg - \xi|^2 - c_1 \frac{(|\yg|^2 - |\xi|^2)^2}{\varepsilon^2 \, d(\ggg)^2}} \, d\xi \lesssim  \frac{ p(\ggg) \, d(\ggg)^{n'} \, e^{ \frac{ d(\ggg)^2 }{ 4 }}}{\bV(i\theta) } \lesssim \int_{\R^n} \!\! e^{-\frac{1}{2}|\yg - \xi|^2 - c_2 \frac{(|\yg|^2 - |\xi|^2)^2}{\varepsilon^2 \, d(\ggg)^2}} \, d\xi,
    \end{align}
    whenever $\ggg \in \R_x^{n} \times \R_u^{n'}$ satisfy the assumption (AN) with $d(\ggg) \ge C$ and $0 < \varepsilon \le \delta_0$.
\end{corollary}

    Substituting $\Yy = \yg, \, \ss= \varepsilon^2 \, d(\ggg)^2$ and $ p = n$ into \eqref{estmate_Ip} below yields the desired estimate \eqref{Eq_1+a_small}. In particular, when $1 + a \ge \varepsilon ^2$, the expression of the estimate can be further simplified into \eqref{Eq_1+a_gtr_epsl}.

    \begin{lemma} [{\cite[Proposition~4.1]{Li12} and 
    \cite[Lemma~7.1]{LZ2025}}]
    Let $p = 2, 3, \ldots$, $\Yy \in \R^p$, $\ss > 0$, and 
    \begin{equation}\label{Eq_def_I}
        \Ic_p(\Yy, \ss) = \int _{\R^p} e^{-\frac{1}{2}|w - \Yy|^2}\exp \left(  -\frac{\left(|w|^2 - |\Yy|^2\right)^2}{\ss}\right) \d w. 
    \end{equation}
    Then we have uniformly for all $\Yy \in \R^p$ and $\ss > 0$ that 
    \begin{align} \label{estmate_Ip}
            \Ic_p(\Yy, \ss ) \sim_p \frac{\sqrt{\ss }}{1 + |\Yy | + \sqrt{\ss}} \, \left( \frac{|\Yy |^2 + \sqrt{\ss }}{1 + |\Yy |^2 + \sqrt{\ss }} \right)^{\frac{p}{2}-1}.
        \end{align} 
    \end{lemma}

    \begin{remark}
    (1) In fact, the estimate \eqref{estmate_Ip} is also valid when $p=1$.

    (2) We point out that \eqref{estmate_Ip} plays an important role in \cite{Li12, LZ19, LZ2025}. In \cite{LTY26}, we will provide its extension when $p = 2$ in order to establish the precise heat kernel bounds on K-type groups. 
    \end{remark}

    \section{\texorpdfstring{Small-time asymptotic behaviour of $p _h (\ggg)$}{}}\label{Sn6}

Once we have in hand the uniform asymptotic behaviour at infinity of the heat kernel, namely Theorems \ref{Thm_est_of_p_epsl_gtr_1}  and \ref{Thm_est_of_K}, we can deduce small-time asymptotic behaviour as in \cite[\S~1.4]{Li12} and \cite[\S~5.4 and Appendix A]{LZ2025}. Here, we only present our final conclusions; the proof is left to interested readers.

\begin{proposition} \label{prop_sta}
         For $ g = (x, u), g'=(x', u') $ that satisfy $x = -x'$ and $  (0 < ) \  2 \, r \geq \pi \, |x|^2$, we have, as $h \rightarrow 0+$,
         \begin{align} \label{eq_stasymp}
             p_h(\ggg) = \frac{ (2\pi) ^ {\frac{n'}{2} }} { h^{ \frac{n'}{2}+n } } \,  \cpi \, \sV_2(i\pi) \, e^{-\frac{d(\ggg)^2}{4h}} \int_{\mathbb{R}^n} \frac{e^{-\frac{1}{2 h} \, \EEg }}{\sqrt{\det \AAg}} \, d \eta \left(1+\O(h^{1/8}) \right),
         \end{align}
         where
\begin{gather*}
    \sV_2(s) := \left( 1 + \frac{ s^2 }{ \pi^2 }\right)^{\frac{ n }{ 2 }} \left( \frac{ s }{ \sinh s }\right)^{ \frac{n}{2} } = \prod_{ k \geq 2 } \left( 1 + \frac{ s^2 }{ k^2 \pi^2 } \right)^{ - \frac{n}{2} }, \quad \mbox{so} \ \sV _2 ( i \pi) = 2 ^{ \frac n 2}, \\
   \det \AAg = \left( \frac{ |\eta|^2 }{ \pi^2 } + \frac{ R^2 }{ 8 } \right)^{n' - 1} \left( \frac{ |\eta|^2 }{ \pi^2 } + \frac{ R^2 }{ 4 } \right), \quad \EEg = \frac{ \Big( |\eta|^2 - \frac{ \pi }{ 4 } (2\, r - \pi \, |x|^2 ) \Big)^2 }{ | \eta |^2 + \frac{ \pi^2 R^2}{ 4 } }. 
\end{gather*}
     \end{proposition}

As a direct consequence, we get 

\begin{corollary}
    (i) If $x = -x' \ne 0$ and $ 2 \, r = \pi \, |x|^2$, then 
    \begin{equation}
    \begin{aligned}
        p _ h (\ggg) = \frac{ 2^{- \frac{ 5 n }{4} - \frac{ 3 }{ 2 }} \, \Gamma ( \frac n 4  )}{ \pi^{ \frac{ n' } 2 } \,  \Gamma ( \frac n 2) } \, h^{ - \frac {3 n} {4} - \frac{ n' }{ 2 } } \, R^{\frac n 2 - n' } \, e^{ - \frac{d (\ggg) ^2 }{ 4 \, h} } \Big( 1 + \O ( h ^ \frac 1 8 ) \Big), \quad h \to 0^+ .
    \end{aligned}
    \end{equation}

    (ii) For $x = -x' \ne 0$ and $ 2 \, r > \pi \, |x|^2$, it holds, as  $h \to 0 ^+$, that:
    \begin{equation}
        \begin{aligned}
            p_h (g, g') = \frac{ 2 ^{ - 2 n - n' +2} }{ \Gamma (\frac n 2 ) } \, h^{- n - \frac{ n' - 1 }{2}} \, \Big( 2r- \frac \pi 2 R^2 \Big) ^{\frac n2 - 1} \, r ^{- \frac{ n ' - 1 }{2}} \, e^{ - \frac{d (\ggg) ^2 }{ 4 \, h} } \, \Big( 1 + \O ( h ^ \frac 1 8 ) \Big). 
        \end{aligned}
    \end{equation}

    (iii) In the case where $x = x' = 0$ and $r > 0$, we have:
    \begin{equation}
    \begin{aligned}
        p_h (\ggg) =  \frac{ 2 ^{ - \frac 3 2 n - n' + 1}}{ \Gamma (\frac n 2 ) } \, h^{- n - \frac{ n' - 1 }{2}} \, r ^{\frac{ n - n ' - 1 }{2}} \, e^{ - \frac{d (\ggg) ^2 }{ 4 \, h} } \, \Big( 1 + \O ( h ^ \frac 1 8 ) \Big), \quad h \to 0 ^+.
        \end{aligned}
    \end{equation} 
\end{corollary}
Notice that the case where $x = x' = 0$ and $ r = 0$ is trivial and the conclusion follows directly from the heat kernel formula. For the remaining case $x \ne -x'$, or $x = - x' \neq 0$ with $2 \, r < \pi \, |x|^2$, Theorem \ref{Thm_est_of_p_epsl_gtr_1} implies
    \begin{equation}
    p_h (\ggg) = \frac{(4 \pi ) ^{- \frac{n}{2} - n'} \, (8 \pi)^\frac{n'}{2} \, \bV(i \theta )}{ \sqrt{ \det ( - \hess_\theta \, \reffun (\theta)  ) } } \, h^{-\frac{ n + n' }{2}} e^{ - \frac{d (\ggg)^2}{ 4 \, h } } \Big( 1 + \O ( h ^ \frac18 ) \Big), \qquad h \to 0 ^ +,
    \end{equation}
    where it follows from Schur's lemma that
    \begin{equation}
    \det ( - \hess _ \theta \, \reffun (\theta) ) = R^{ 2 n'} \, \Big( \frac{ \ma ( |\theta| ) }{ | \theta | } \Big) ^{n' - 1 } \, \ma' ( |\theta| ) .
\end{equation}

The above result generalizes the counterpart in the case where $n' = 1$ and $n \ge 3$ obtained in \cite[\S~1.4]{Li12}. Furthermore, under the assumption in Proposition \ref{prop_sta}, it holds that
$d(\ggg)^2 = \wtreffun( \ggg ; i \, \pi \frac{u - u'}{r} )$, and we can characterize all shortest geodesics from $g$ to $g'$ by means of $\pi \frac{u - u'}{r}$ and $\{ \eta \in \R^n; \ |\eta|^2 =  \frac{ \pi }{ 4 } (2\, r - \pi \, |x|^2 )\}$. As a direct consequence, it readily determines the cut locus.

\section{\texorpdfstring{Spatial derivative estimates for $p( \ggg ) $}{}} \label{Sn7} 

In this section, we only prove Theorem \ref{thm_deriv} in the case $l = 1 $, since the higher-order estimates follow by iteration of the same argument, with no substantial new difficulty. By the scaling property, it suffices to show
\begin{equation} \label{70nf}
    | \nabla _G \, p (\ggg) | \lesssim \Big( 1 + d(\ggg) \Big) \, p(\ggg), \qquad \forall \,g,\, g' \in \R_x^n \times \R_u^{n'}.
\end{equation}
Since the previous results (cf. \eqref{daoshu_up}-\eqref{ineq_derivupbd}) already yield the desired bound when $d(\ggg) \lesssim 1$, it remains to treat $d(\ggg) \gg 1$. We further divide the discussion into two cases: $\varepsilon\ge \delta_0$ and $\varepsilon<\delta_0$.

We begin with an elementary estimate for the modified amplitude produced by differentiation. Recall that the meaning of $|\lambda| \csch{|\lambda|}$ and $|\lambda| \coth{|\lambda|}$ (for suitable $\lambda \in \C^{n'}$) is given after \eqref{26f}.

\begin{lemma} \label{lem_psi_uppbd}
    Assume $\varepsilon \ge \delta_0$. Then, it holds that: 
    \begin{equation} \label{ineq_61}
        \left| \frac{ |\lambda+ i \theta|}{\sinh{ |\lambda + i \theta | }} - |\lambda + i \theta | \coth{ |\lambda + i \theta| } \right| \lesssim _{\delta_0} \frac{ | \lambda | }{ \sinh |\lambda | } ( 1 + \cosh { | \lambda | } ) \lesssim 1 + | \lambda |,
    \end{equation}
    whenever $\lambda \in \R^{n'}$.
    Moreover, we have 
    \begin{equation} \label{ineq_62}
        \left| \frac{ |\lambda+ i \theta|}{\sinh{ |\lambda + i \theta | }} - |\lambda + i \theta | \coth{ |\lambda + i \theta| } \right| \lesssim _{\delta_0} |\lambda|^2 + |\theta|^2, \qquad \forall \, |\lambda| \lesssim 1.
    \end{equation}
\end{lemma}

\begin{proof}
The proof is straightforward. First, notice that 
\[
 \frac{|\lambda+ i \theta|}{\sinh{ |\lambda + i \theta | }} - |\lambda + i \theta | \coth{ |\lambda + i \theta| } = \frac{|\lambda+ i \theta|}{\sinh{ |\lambda + i \theta | }} \left( 1 - \cosh{|\lambda + i \theta |} \right).
\]
Next, it follows from \eqref{26f} ( or as a special case of  \cite[Lemma 4.1]{LZ2025}) that:
\begin{equation*}
        \left| \frac{ |\lambda + i \theta| }{ \sinh |\lambda + i \theta| } \right| \le \frac{|\theta|}{\sin{|\theta|}} \frac{ |\lambda| }{ \sinh{ |\lambda| } } \lesssim _{\delta_0} \frac{ |\lambda| }{ \sinh{ |\lambda| } }.
    \end{equation*}
Combining this with the elementary inequality $|\cosh{z} - 1| \le \cosh{|z|} - 1$ provided $z \in \C$, we easily conclude the desired claims. 
\end{proof}

With Lemma \ref{lem_psi_uppbd} established, we continue to show upper bounds for the derivatives of $p(\ggg)$. 

\subsection{\texorpdfstring{Proof of \eqref{70nf}}{}} \label{Sn71}

\begin{proof}
    {We assume that $d(\ggg)$ is sufficiently large.} {A direct differentiation of the integral formula \eqref{Eq_def_of_p} yields
    \begin{align}
    X_j\, p( \ggg) &=(4 \pi) ^{-\frac{n}{2}-n'} \!\! \int_{\R^{n'}} \!\! \bV(\lambda) \cdot \left( { \frac 1 2} \, \Xi_j (\lambda) \right)  \, \exp \left\{ -\frac{1}{4} \, \wtreffun(\ggg;\lambda) \right\} \, d \lambda, \label{eq_65} \\
    U_{jk} \, p(\ggg) & = (4 \pi) ^{-\frac{n}{2}-n'} \!\! \int_{\R^{n'}} \!\! \bV(\lambda) \cdot \left( \frac{ i }{ 2} \, \Upsilon_{jk}(\lambda) \right) \, \exp \left\{ -\frac{1}{4} \, \wtreffun(\ggg;\lambda) \right\} \, d \lambda, \label{eq_66}
\end{align}
where
\begin{align}
    \Xi_j(\lambda) =  
     - (|\lambda|\coth|\lambda|) \,x_j + \frac{|\lambda|}{\sinh|\lambda|} \,x'_j, \qquad \Upsilon_{jk}(\lambda) =  x_j \, \lambda_k.
\end{align}}
    {Write
    \begin{equation} \label{eq_69}
        \Xi_j (\lambda) =  \frac{ | \lambda | }{ \sinh |\lambda| } \, ( x_j' - x_j) + \left( \frac{ |\lambda| }{ \sinh{ |\lambda| } } - | \lambda | \coth{ |\lambda| } \right) \, x_j.
    \end{equation}
    Accordingly, we split \eqref{eq_65} into two terms, say $\mathsf{I}_1$ and $\mathsf{I}_2$.}

    We first turn to the easy case $\varepsilon\ge \delta_0$. Recall from \eqref{Eq_expression_of_d} that $d(\ggg)^2 \gtrsim | x - x' |^2 + R^2 \, |\theta|^2$. After shifting the contour from $\mathbb R^{n'}$ to $\mathbb R^{n'}+i\theta$, we may obtain by a stationary phase argument and Lemma \ref{lem_psi_uppbd} that 
    \begin{equation*}
    |X_j \, p( \ggg )|+|U_{jk} \, p( \ggg )|
    \lesssim
    d (\ggg) \,p( \ggg ).
    \end{equation*}

It remains to consider the difficult case $\varepsilon < \delta_0$. For $X_j \, p(g,g')$, we differentiate the representation \eqref{Eq_transformed_1_of_p}. This produces two terms: one coming from differentiation of the Gaussian factor in the $s$-variable, and the other from differentiation of $\wtff (R,a;\lambda)$ with respect to $x_j$. Specifically,
\begin{align*}
    X_j \, p(\ggg) & = \cpi \int_{\R^n} \! ( s_j - (\xg)_j ) \, e^{ -\frac{1}{2} |s-\xg|^2} \d s \!
    \int_{\R^{n'}} \! \bVt(\lambda) \, e^{-\frac{ \wtff( \ra ; \lambda ) - 2 \, i \, (u - u')\cdot \lambda }{4} -\frac{ |\lambda|^2 |s|^2}{ 2 \, \pi^2 }}  \d \lambda \\
    & \quad - \cpi \int_{\R^n} \!\! e^{ -\frac{1}{2} |s-\xg|^2 } \d s 
    \int_{\R^{n'}} \!\! \Big( \frac 1 4 \partial_{x_j} \, \wtff ( \ra ; \lambda ) \Big) \bVt(\lambda) \, e^{-\frac{ \wtff( \ra ; \lambda ) - 2 \, i \, (u - u')\cdot \lambda }{4}  -\frac{ |\lambda|^2 |s|^2}{ 2 \, \pi^2 } }  \d \lambda.
\end{align*}
For the first term, we perform the same contour deformation and change of variables as in the proof of Proposition \ref{Thm_transform}. The only new factor is bounded by \(C \, d(g,g')\) in \(\rgn_0\) or \(\adr_0\), while outside the region of main contribution the exponential decay estimates from Section \ref{Sn5} dominate. Hence, this term contributes at most $ C \, d (\ggg)\,p( \ggg )$. For the second term, a direct computation shows that
\begin{equation*}
|\partial_{x_j} \, \wtff (R,a;\lambda+i\theta)| \lesssim R \, (1+|\lambda|^2).
\end{equation*}
Now that $R\lesssim d(\ggg)$ by \eqref{Eq_esti_of_d}, the contour deformation and the corresponding estimates established in Section \ref{Sn5} again apply. This yields
\[
|X_j \, p(g,g')| \lesssim d (\ggg) \, p( \ggg).
\]
The estimate for $U_{jk} \, p(g,g')$ is derived in an analogous way, since after the contour shift the additional factor \(x_j \, ( \lambda_k + i \theta_k ) \) is controlled by \( \pi R \, ( 1 + |\lambda| ) \). Thus, we omit the details.

Combining the estimates for $X_j \, p( \ggg )$ and $U_{jk} \, p( \ggg )$, we obtain
\[
|\nabla_G \, p(g,g')| \lesssim d (\ggg) \, p( \ggg ),
\]
which completes the proof of Theorem \ref{thm_deriv}.
\end{proof}

\section*{Acknowledgement}
	\addcontentsline{toc}{section}{Acknowledgement}
	
	~~~~~~This work is partially supported by NSF of China (Grant No. 12271102).

\normalem
\phantomsection\addcontentsline{toc}{section}{References}
\bibliographystyle{abbrv}
\bibliography{references} 

\noindent
\\
{\small\it\\
\noindent
Yimeng Chen, Hong-Quan Li, Jun-Cheng Tang, Jia-Yu Yang}\\
{\small\it   School of Mathematical Sciences,
Fudan University}\\
{\small\it  220 Handan Road,
Shanghai 200433 China}\\
{\small\it E-mail:\;\;\tt ymchen24@m.fudan.edu.cn, hongquan\_li@fudan.edu.cn, 
25110180045@m.fudan.edu.cn,
yangjy25@m.fudan.edu.cn}\\

\end{document}